\documentclass[11pt]{amsart}

\usepackage{amsfonts, amssymb, amsthm}

\usepackage{arydshln}
  \usepackage{graphicx}

\usepackage[foot]{amsaddr}

\usepackage[dvipsnames]{xcolor}
\usepackage{hyperref}

\theoremstyle{definition}
\newtheorem{theorem}{Theorem}
\newtheorem{lemma}[theorem]{Lemma}

\newtheorem{corollary}[theorem]{Corollary}
\newtheorem{definition}[theorem]{Definition}
\newtheorem{remark}[theorem]{Remark}

\newtheorem{example}[theorem]{Example}
\newtheorem{notation}[theorem]{Notation}

\DeclareMathOperator{\End}{End}

\usepackage{enumitem}

\usepackage{comment}

 \def\<{\langle}
  \def\>{\rangle}

\newcommand{\Hdeg}{\mathrm{H}^{\mathrm{deg}}} 
\newcommand{\sv}{\mathop{\mathrm{sV}\hspace{-2.5mm}\mathrm{V}}\nolimits} 

\newcommand{\ab}{(\underline{a}, \underline{b})}

\newcommand{\K}{\mathbb{C}}
\newcommand{\N}{\mathbb{N}}
\newcommand{\diag}{\mathrm{diag}}

\newcommand{\Id}{\mathrm{Id}}
\newcommand{\GL}{\mathrm{GL}}

\allowdisplaybreaks

\begin{document}

\title[On calibrated representations of $\sv_2$]{On calibrated representations of the degenerate affine periplectic Brauer algebra}
 
\author{Zajj Daugherty}
\address[Z.D.]{Department of Mathematics, City College of New York, New York, NY 10031 USA}
\email{zdaugherty@gmail.com}
\author{Iva Halacheva}
\address[I.V.]{School of Mathematics and Statistics, University of Melbourne, Parkville VIC 3010
    Australia}
\email{iva.halacheva@unimelb.edu.au}
\author{Mee Seong Im}
\address[M.S.I.]{Department of Mathematical Sciences, United States Military Academy, West Point, NY 10996 USA}
\email{meeseongim@gmail.com}
\author{Emily Norton}
\address[E.N.]{Max Planck Institute for Mathematics, 53111 Bonn
Germany}
\email{enorton@mpim-bonn.mpg.de}



\begin{abstract}
We initiate the representation theory of the degenerate affine periplectic Brauer algebra on $n$ strands by constructing its finite-dimensional calibrated representations when $n=2$. We show that any such representation that is indecomposable and does not factor through a representation of the degenerate affine Hecke algebra occurs as an extension of two semisimple representations with one-dimensional composition factors; and furthermore, we classify such representations with regular eigenvalues up to isomorphism.
\end{abstract}  
\maketitle

\section*{Introduction}

The degenerate affine periplectic Brauer algebra on $n$ strands, or $\sv_n$ for short, belongs to a family of diagram algebras playing various roles in generalized Schur-Weyl dualities. Related algebras include the Brauer algebra, periplectic Brauer algebra, degenerate affine Hecke algebra, Nazarov--Wenzl algebra, and walled Brauer algebra. The algebra $\sv_n$ was first defined by Chen and Peng by generators and relations  \cite{ChenPeng}\footnote{where it was called ``affine periplectic Brauer algebra"}, and in previous work of the authors together with other collaborators as the endomorphism algebra of the object $n$ in a certain monoidal supercategory \cite{us+}\footnote{where it was called ``affine VW superalgebra"}. That monoidal supercategory arises from the representation theory of the periplectic Lie superalgebra, hence the word ``periplectic," while ``degenerate affine" indicates the close relationship of $\sv_n$ with the degenerate affine Hecke algebra $\Hdeg_n$, which is a quotient of $\sv_n$. Like $\Hdeg_n$, the algebra $\sv_n$ contains a large polynomial subalgebra $\K[y_1,\dots,y_n]$ which provides a point of leverage for its representation theory.

 In the work at hand, we begin the representation theory of the algebras $\sv_n$ with the smallest nontrivial example of these algebras, namely $\sv_2$.  
Our goal in this paper is to explicitly construct finite-dimensional \textit{calibrated representations}, that is, representations of $\sv_2$ on which the polynomial subalgebra $\K[y_1,y_2]$ acts diagonalizably. Our approach here is very concrete: write down matrices for the action of the generators of $\sv_2$, find conditions on these matrices for the representation to be indecomposable, and determine when two such indecomposable representations are isomorphic. 

The representations of $\sv_2$ that we focus on are the ones that cannot be obtained as representations of $\Hdeg_2$ by setting the Temperley-Lieb type generator $e$ equal to $0$, since calibrated representations of (degenerate) affine Hecke algebras in small rank are known by work of Ram \cite{Ram}.
In Section \ref{calibrated} we give a recipe for producing ``new" calibrated representations of $\sv_2$, i.e. ones on which $e$ does not act by $0$. Then we show that our recipe produces all such indecomposable calibrated representations; this is Theorem \ref{main1}. Theorem \ref{main1} implies that an indecomposable finite-dimensional calibrated representation with nonzero action of $e$ always occurs as an extension of two semisimple representations with $1$-dimensional composition factors (Corollary \ref{ext}).
In Theorems \ref{main2}, \ref{iso1}, and \ref{iso2} we completely classify the indecomposable finite-dimensional calibrated representations of $\sv_2$  up to isomorphism on which $y_1$ and $y_2$ act with regular eigenvalues. In addition to the eigenvalues, the other classifying device turns out to be an unexpected yet natural class of matrices that we name rhizomatic, see Section \ref{rhizomatic}.

We expect that some of the ideas in this paper will generalize to the case $n>2$, but the algebras $\sv_n$ for $n>2$ are considerably more complicated, so we also expect that more work and possibly more ideas will be needed to deal with their calibrated representations.

\section{Definitions}

The degenerate affine Hecke algebra $\Hdeg_n$ was introduced by Drinfeld \cite{Drinfeld} and Lusztig \cite{Lusztig1}. It contains $\K[y_1,\dots,y_n]$ and $\K S_n$ as subalgebras, and together they generate $\Hdeg_n$. We recall its generators and relations in the case $n=2$.

\begin{definition}\cite{Drinfeld} The degenerate affine Hecke algebra $\Hdeg_2$ is the $\K$-algebra generated by $s$, $y_1$, and $y_2$ with relations:
$$s^2=1,\quad y_1y_2=y_2y_1,\quad sy_1=y_2s-1, \quad sy_2=y_1s+1.$$
\end{definition}
\noindent Multiplying both sides of the third relation by $s$ we get the fourth relation and vice versa, but it can be convenient to use this bigger set of relations. 

\begin{definition} \cite[Definition 39]{us+}
The degenerate affine periplectic Brauer algebra $\sv_2$ is the $\K$-algebra generated by $s$, $y_1$, $y_2$, and $e$ with relations:
\begin{align*}
&s^2=1,\quad y_1y_2=y_2y_1,\quad sy_1=y_2s-1-e, \quad sy_2=y_1s+1-e,\\ 
&\quad e^2=0,\quad es=e,\quad se=-e,\quad ey_2=ey_1+e,\quad y_1e=y_2e+e. 
\end{align*}
\end{definition}
\noindent Again, this is not a minimal set of generators and relations but it is convenient to use this bigger set. Notice that $e$ is generated by $y_1$, $y_2$, and $s$. It follows from the relations given that $ef(y_1,y_2)e=0$ for any polynomial $f(y_1,y_2)\in\K[y_1,y_2]$, see \cite[Lemma 48]{us+}.

We cannot hope to classify indecomposable representations of $\Hdeg_2$ or $\sv_2$ in general. However, we may hope to classify a well-behaved subset of indecomposable representations: those finite-dimensional indecomposable representations on which $y_1$ and $y_2$ act diagonalizably. 

\begin{definition} Let $H$ be $\Hdeg_2$ or $\sv_2$. A representation $V$ of $H$ is called \textit{calibrated} if $V$ has a basis with respect to which the actions of  $y_1$ and $y_2$ on $V$ are given by diagonal matrices.

\end{definition}

\begin{notation}
We denote by $M_{m\times n}(\K)$ the ring of $m \times n$ matrices with entries in $\K$. We write $\K^{k+\ell}$ for the $\K$-vector space of dimension $k+\ell$ whose vectors $\ab$ are viewed as the concatenation of a vector $\underline{a}=(a_1,\dots,a_k)$ of length $k$ and a vector $\underline{b}=(b_1,\dots,b_\ell)$ of length $\ell$. 
\end{notation}

\section{Calibrated representations of $\sv_2$}\label{calibrated}
In this section we construct the calibrated representations of $\sv_2$. The starting point is the obvious relationship to the degenerate affine Hecke algebra:
\begin{lemma}\label{degHecke} Let $V$ be a representation of $\sv_2$ on which $e$ acts by $0$. Then the action of $\sv_2$ on $V$ factors through $\Hdeg_2\cong \sv_2/\langle e\rangle$. Conversely, if $W$ is a representation of $\Hdeg_2$ then we may extend $W$ to a representation of $\sv_2$ by declaring $e$ to act by $0$.
\end{lemma}
\noindent The calibrated representations of the affine Hecke algebra are known by work of Ram \cite{Ram}. As remarked by Suzuki in his study of $\Hdeg_n$-representations \cite{Suzuki}, Lusztig's  work \cite{Lusztig1},\cite{Lusztig2} shows that the representation theory of degenerate affine Hecke algebras and affine Hecke algebras can be recovered from each other. Calibrated representations of $\Hdeg_2$ may therefore be considered as known.
To classify calibrated representations of $\sv_2$, we then need to construct those on which $e$ does not act by $0$. We will do this by deforming certain calibrated representations of $\Hdeg_2$.
\begin{definition}
For any $a\in\K$, let $V_a^+$ be the one-dimensional $\Hdeg_2$-representation on which $y_1$ acts by multiplication by $a$ and $s$ acts trivially; let $V_a^-$ be defined similarly except $s$ acts by multiplication by $-1$.
\end{definition}
\noindent Using Lemma \ref{degHecke} and observing that $e^2=0$ forces $e$ to act by $0$ on any one-dimensional representation, we have:
\begin{lemma}
The one-dimensional representations of $\sv_2$ are exactly $\{V_a^+,\; V_a^-\mid a\in\K\}$.
\end{lemma}

Now let $k, \ell\in \N$. Let $S$ be any $k\times \ell$ matrix. Then there is a $(k+\ell)$-dimensional calibrated $\Hdeg_2$-representation $W_{k,\ell}(S)$ which fits into a short exact sequence:

$$ 0\rightarrow (V_0^-)^{\oplus k} \rightarrow W_{k,\ell}(S)\rightarrow (V_{-1}^+)^{\oplus \ell}\rightarrow 0,$$
where $y_1$ acts on $W_{k,\ell}(S)$ by the diagonal matrix $Y_1$ with $0$'s in the first $k$ diagonal entries and $-1$'s in the last $\ell$ diagonal entries, and $s$ acts by the block matrix $\widetilde{S}=\begin{pmatrix}-\Id_k & S\\ 0 & \Id_\ell\end{pmatrix}$. Using the relation $sy_1s+s=y_2$, we get that $y_2$ acts on $W_{k,\ell}(S)$ by the diagonal matrix $Y_2$ with $-1$'s in the first $k$ diagonal entries and $0$'s in the last $\ell$ diagonal entries.

\begin{example} Let $k=3$ and $\ell=2$, and let $S=\begin{pmatrix}s_{11} &s_{12}\\s_{21}&s_{22}\\s_{31}&s_{32}\end{pmatrix}$. Then the actions of $y_1$, $y_2$, and $s$ on $W_{3,2}(S)$ are given by the following matrices:

$$\small 
Y_1=
\left(  
\begin{array}{ccc;{1pt/1pt}cc}
0&0&0&0&0\\
0&0&0&0&0\\
0&0&0&0&0\\
\hdashline[1pt/1pt]
0&0&0&-1&0\\
0&0&0&0&-1
\end{array}
\right),  
\quad 
Y_2=
\left( 
\begin{array}{ccc;{1pt/1pt}cc}
-1&0&0&0&0\\
0&-1&0&0&0\\
0&0&-1&0&0\\
\hdashline[1pt/1pt]
0&0&0&0&0\\
0&0&0&0&0
\end{array}
\right), 
\quad
\widetilde{S}=
\left( 
\begin{array}{ccc;{1pt/1pt}cc}
-1&0&0&s_{11}&s_{12}\\
0&-1&0&s_{21}&s_{22}\\
0&0&-1&s_{31}&s_{32}\\
\hdashline[1pt/1pt]
0&0&0&1&0\\
0&0&0&0&1
\end{array}
\right). 
$$
\end{example}

Now let $\ab:=(a_1,a_2,\dots,a_k,b_1,b_2,\dots,b_\ell)\in\K^{k+\ell}$ be any $(k+\ell)$-tuple of complex numbers. We may build a calibrated representation of $\sv_2$ from $W_{k,\ell}(S)$ and $\ab$.
\begin{lemma}
 Let $V_{k,\ell}(S;\ab)$ be a $(k+\ell)$-dimensional $\K$-vector space and consider the following matrices in $\End_\K(V_{k,\ell}(S;\ab)$:
 $$y_1=Y_1+\diag\ab, \quad y_2=Y_2+\diag\ab, \quad s=\widetilde{S},\quad e=-sy_2+y_1s+\Id_{k+\ell}.$$ 
 Then $V_{k,\ell}(S;\ab)$ is a calibrated representation of $\sv_2$ on which $y_1,y_2,s,e$ act by the matrices with the same names.
\end{lemma}
\begin{proof} The matrix $e$ is a block matrix $e=\begin{pmatrix}0&E\\0&0\end{pmatrix}$
where $E$ is a $k\times \ell$ matrix with entries  $e_{ij}=(a_i-b_j)s_{ij}.$ It is then a straightforward computation with matrices to check that the defining relations of $\sv_2$ are satisfied.\end{proof}

\begin{example}
For $k=3$ and $\ell=2$ the matrices look like:
\begin{align*}&y_1=
\left( 
\begin{array}{ccc;{1pt/1pt}cc}
a_1&0&0&0&0\\
0&a_2&0&0&0\\
0&0&a_3&0&0\\
\hdashline[1pt/1pt]
0&0&0&b_1-1&0\\
0&0&0&0&b_2-1
\end{array}
\right), 
\quad
y_2=
\left(
\begin{array}{ccc;{1pt/1pt}cc}
a_1-1&0&0&0&0\\
0&a_2-1&0&0&0\\
0&0&a_3-1&0&0\\
\hdashline[1pt/1pt]
0&0&0&b_1&0\\
0&0&0&0&b_2
\end{array}
\right), 
\\
&s=
\left(
\begin{array}{ccc;{1pt/1pt}cc}
-1&0&0&s_{11}&s_{12}\\
0&-1&0&s_{21}&s_{22}\\
0&0&-1&s_{31}&s_{32}\\
\hdashline[1pt/1pt]
0&0&0&1&0\\
0&0&0&0&1
\end{array}
\right), 
\quad
e=
\left( 
\begin{array}{ccc;{1pt/1pt}cc}
0&0&0&(a_1-b_1)s_{11}&(a_1-b_2)s_{12}\\
0&0&0&(a_2-b_1)s_{21}&(a_2-b_2)s_{22}\\
0&0&0&(a_3-b_1)s_{31}&(a_3-b_2)s_{32}\\
\hdashline[1pt/1pt]
0&0&0&0&0\\
0&0&0&0&0
\end{array}
\right). 
\end{align*}
\end{example}

We can then think of the family of calibrated representations of $\sv_2$ constructed by this procedure as being parametrized by pairs consisting of a $\Hdeg_2$-representation $W_{k,\ell}(S)$ as above together with a vector $\ab\in\K^{k+\ell}$; equivalently, by pairs $(S,\ab)$ consisting of a $k\times \ell$ matrix $S\in M_{k\times \ell}(\K)$ and a vector $\ab\in\K^{k+\ell}$. When we take $\ab$ to be the $0$-vector, then $e$ is the $0$ matrix, $y_1=Y_1$, $y_2=Y_2$, $s=\widetilde{S}$, and so we get back the representation $W_{k,\ell}(S)$ of $\Hdeg_2$. 
Note that nonzero choices of $\ab$ may produce representations on which $e$ acts by $0$: for example, taking $\ab=(a,\dots,a,a,\dots,a)$ for any $a\in\K$ forces $e=0$. This choice of $\ab$ has the effect of shifting the eigenvalues by which $y_1$ and $y_2$ act by $a$.

 \subsection{The shape of calibrated representations when $e$ does not act by $0$} 
The next step in our classification of calibrated $\sv_2$-representations consists in showing that all calibrated representations on which $e$ does not act by $0$ arise via the construction just given in the preceding subsection. We will often abuse notation and give the matrices representing the generators the same names as the generators of the abstract algebra themselves.
 \begin{theorem}\label{main1}
 Suppose $V$ is a finite-dimensional, indecomposable calibrated representation of $\sv_2$ on which $e$ does not act by $0$. Then $V=V_{k,\ell}(S;\ab)$ for some $k,\ell\in\N$, some $S\in M_{k\times\ell}(\K)$, and some $\ab\in\K^{k+\ell}$.
 \end{theorem}
 \begin{proof}
By assumption there is a basis for $V$ such that $y_1$ and $y_2$ act by diagonal matrices. We can choose this basis so that the matrix for $y_1-y_2$ has the form $y_1-y_2=\diag(1,\dots,1,-1,\dots,-1,d_1,\dots,d_1,-d_1,\dots,-d_1,\dots, d_s,\dots,d_s,-d_s,\dots,-d_s)$, where $1$ occurs $k$ times and $-1$ occurs $\ell$ times, and say $d_i$ occurs $k_i$ times, $-d_i$ occurs $\ell_i$ times. Let $e=(e_{ij})$ be the matrix of $e$. Using the relations $e=(y_1-y_2)e$ and $-e=e(y_1-y_2)$ and writing out the equations for the matrix entries $e_{ij}$, we see that $e$ has all $0$ entries except for in the $k\times \ell$ block $ 1\leq i\leq k,\; k+1\leq j\leq k+\ell$.
Next, we look at the shape of the matrix of $s$. Adding the two equations mixing $s$ and the $y_i$'s, we have the equation $s(y_1-y_2)+(y_1-y_2)s=-2$. Solving this equation for the matrix entries of $s$, we see that $s$ is a block matrix with blocks
$\begin{pmatrix} -\Id_k & S\\ T & \Id_\ell 
\end{pmatrix} $
in the upper left corner; then arranged down the diagonal, further square blocks of shape 
$\begin{pmatrix}
-\frac{1}{d_i}\Id_{k_i}& S_i\\T_i & \frac{1}{d_i}\Id_{\ell_i}
\end{pmatrix}$
and $0$'s everywhere else, 
giving us


\[ 
\small
s = 
\left( 
\begin{array}{cc;{1pt/1pt}cc;{1pt/1pt}cc;{1pt/1pt}cc}
-\Id_k & S & &0& & 0 & &0 \\ 
T & \Id_{\ell} & & & & & & \\ 
\hdashline[1pt/1pt]
& &-\frac{1}{d_1}\Id_{k_1} & S_1& & & & \\ 
0& & T_1& \frac{1}{d_1} \Id_{\ell_1}& & 0& &0 \\ 
\hdashline[1pt/1pt]
& & & & \ddots& & & \\ 
0& & &0 & & \ddots& &0 \\ 
\hdashline[1pt/1pt]
& & & & & &-\frac{1}{d_s}\Id_{k_s} &S_s \\ 
0 & & &0 & & 0& T_s& \frac{1}{d_s} \Id_{\ell_s}\\ 
\end{array}
\right). 
\]  


Now, considering the shapes of $y_1$, $y_2$, $e$, and $s$, we observe that they are all block matrices with 
\begin{itemize}
\item a $(k+\ell) \times (k+\ell)$ block in the upper left corner;
\item a $(\sum_i( k_i+\ell_i))\times (\sum_i (k_i+\ell_i))$ block in the lower right corner;
\item blocks made of $0$'s in the upper right and lower left corner.
\end{itemize}
It follows that the representation $V$ is the direct sum $V=V_1\oplus V_2$ where $V_1$ is $(k+\ell)$-dimensional and the action of $y_1$, $y_2$, $e$, and $s$ on $V_1$ is given by the matrix block of size $(k+\ell) \times (k+\ell)$ in the upper left corner, and where $e$  acts by $0$ on $V_2$. Since $V$ is an indecomposable representation on which $e$ does not act by $0$, $V=V_1$.

Write $y_1=\diag(a_1,\dots,a_k,b_1-1,\dots,b_\ell-1)$ and $y_2=\diag(a_1-1,\dots,a_k-1,b_1,\dots,b_\ell)$. Using the relations $es=e$, $se=-e$, and $s^2=1$ gives the following information about $e$ and $s$:
$$s=\begin{pmatrix}
-\Id_k & S\\ T & \Id_\ell
\end{pmatrix},
\qquad
e=\begin{pmatrix} 0 & E \\ 0 & 0
\end{pmatrix},
$$
where 
$S=(s_{ij})$, $T=(t_{ji})$, $E=(e_{ij})$, $1\leq i\leq k$ and $1\leq j\leq \ell$, satisfy the equations \begin{equation}\label{T}TS=ST=ET=TE=0,\qquad e_{ij}=(a_i-b_j)s_{ij}, \qquad t_{ji}(b_j-a_i)=0.
\end{equation}

Suppose $T\neq 0$. We will show that $V$ is decomposable. Since we are assuming $e\neq 0$, we also have $S\neq 0$. 
Let $\mathbf{v}_i$ be an eigenvector for $y_1$ with eigenvalue $a_i$ and let $\mathbf{w}_j$ be an eigenvector for $y_1$ with eigenvalue $b_j-1$. By assumption, $0\subsetneq\mathrm{Im}(T)\subseteq \mathrm{Ker}(S)\subsetneq  \K^\ell\cong \mathrm{Span}_\K(\mathbf{w}_1,\dots,\mathbf{w}_\ell)$ and $0 \subsetneq\mathrm{Im}(S)\subseteq \mathrm{Ker}(T)\subsetneq \K^k\cong\mathrm{Span}_\K(\mathbf{v}_1,\dots,\mathbf{v}_k)$. 
Under the isomorphism $\mathrm{Span}_\K(\mathbf{v}_1,\dots,\mathbf{v}_k)\cong\K^k$ we identify $\mathbf{v}_i=(0,\dots,0,1,0,\dots,0,0,\dots,0)$ with the vector $\overline{\mathbf{v}}_i:=(0,\dots,0,1,0,\dots,0)$ where we delete the last $\ell$ zeros from $\mathbf{v}_i$; similarly under the isomorphism $\mathrm{Span}_\K(\mathbf{w}_1,\dots,\mathbf{w}_\ell)\cong \K^\ell$ we drop the first $k$ zeros from the vector $\mathbf{w}_\ell$ and call the resulting vector $\overline{\mathbf{w}}_\ell$. These isomorphisms are obviously equivariant for the $y_1$ and $y_2$ actions, where $y_1$ acts by $\diag(a_1,\dots,a_k)$ on $\K^k$ and by $\diag(b_1-1,\dots,b_\ell-1)$ on $\K^\ell$, and similarly with $y_2$.. Since $T(\overline{\mathbf{v}}_i)$ is just the $i$'th column of $T$, by Equation \eqref{T} it follows that for any $1\leq j\leq \ell$ such that $t_{ji}\neq 0$, $b_j=a_i$ and thus $y_1\mathbf{w}_j=(a_i-1)\mathbf{w}_j$. So $y_1T(\overline{\mathbf{v}}_i)=y_1\sum_{j=1}^\ell t_{ji}\overline{\mathbf{w}}_j=\sum_{j=1}^\ell t_{ji}(b_j-1)\overline{\mathbf{w}}_j=(a_i-1)\sum_{j=1}^\ell t_{ji}\overline{\mathbf{w}}_j=(a_i-1)T(\overline{\mathbf{v}}_i)$.
This shows that $\mathrm{Im}(T)$ consists of eigenvectors for $y_1$. If we take $\sum_{j=1}^\ell f_j\overline{\mathbf{w}}_j\in\mathrm{Im}(T)^\perp$, a vector space complement to $\mathrm{Im}(T)$ in $\K^\ell$, then $T(\overline{\mathbf{v}}_i)\cdot (y_1\sum_{j=1}^\ell f_j\overline{\mathbf{w}}_j)=\sum_{j=1}^\ell t_{ji}(b_j-1)f_j=(a_i-1)T(\overline{\mathbf{v}}_i)\cdot \left(\sum_{j=1}^\ell f_j\overline{\mathbf{w}}_j\right)=0.$
So $y_1$ preserves $\mathrm{Im}(T)^\perp$.

Next, we show that $y_1$ preserves $\mathrm{Ker}(T)$. 
 Again, Equation \eqref{T} shows that $a_m=b_j=a_i$ whenever $t_{jm}\neq 0$ is in the same row as $t_{ji}\neq 0$. Take $\overline{\mathbf{u}}\in\mathrm{Ker}(T)$ and write $\overline{\mathbf{u}}=\sum_{i=1}^k c_i\overline{\mathbf{v}}_i$ for some $c_i\in\K$. 
Fix a row $\mathbf{t}_j$ of $T$. Since $t_{ji}=0$ whenever $b_j\neq a_i$, we then have $0=b_j\left(\mathbf{t}_j\cdot\overline{\mathbf{u}}\right)=b_j(\sum_{i=1}^k t_{ji}c_i)=\sum_{i=1}^k t_{ji}b_jc_i=\sum_{i=1}^kt_{ji}a_ic_i=\mathbf{t}_j\cdot(y_1\overline{\mathbf{u}})$, showing that $y_1$ preserves $\mathrm{Ker}(T)$.  Then $y_1\mathrm{Im}(S)\subseteq \mathrm{Ker}(T)$ since $\mathrm{Im}(S)\subseteq\mathrm{Ker}(T)$ and $y_1\mathrm{Ker}(T)\subseteq \mathrm{Ker}(T)$. Let $\mathrm{Ker}(T)^\perp$ be a vector space complement to $\mathrm{Ker}(T)$ in $\K^k$. If $\sum_{i=1}^kd_i\overline{\mathbf{v}}_i=\overline{\mathbf{z}}\in\mathrm{Ker}(T)^\perp $ then $(y_1\overline{\mathbf{z}})\cdot\overline{\mathbf{u}}=\sum_{i=1}^ka_id_ic_i=\overline{\mathbf{z}}\cdot (y_1\overline{\mathbf{u}})=0$ so $y_1$ preserves $\mathrm{Ker}(T)^\perp$ as well. 

All the preceding arguments apply as well to $y_2$ as to $y_1$ since $y_1-y_2=\begin{pmatrix}\Id_k&0\\0&-\Id_\ell\end{pmatrix}$. Now take $V_1$ to be the $\sv_2$-subrepresentation of $V$ generated by a vector space complement to $\mathrm{Ker}(T)$ in $\mathrm{Span}_\K(\mathbf{v}_1,\dots,\mathbf{v}_k)$, and take $V_2$ to be the subrepresentation of $V$ generated by $\mathrm{Ker}(T)$ together with a vector space complement to $\mathrm{Im}(T)$ in $\mathrm{Span}_\K(\mathbf{w}_1,\dots,\mathbf{w}_\ell)$. By construction $s=\begin{pmatrix} -\Id_k & S\\ T & \Id_\ell\end{pmatrix}$ preserves $V_1$ and $V_2$. We checked above that $y_1$ and $y_2$ preserve $V_1$ and $V_2$. Since $e=y_1s-sy_2+1$, $e$ also preserves $V_1$ and $V_2$. Then $V_1\neq 0$, $V_2\neq 0$, $V_1+V_2=V$ and $V_1\cap V_2=0$, and therefore $V\cong V_1\oplus V_2$ is decomposable.
 \end{proof}
 
Theorem \ref{main1} in pictures says that if $V$ is indecomposable and $e$ does not act by $0$, then the matrices of $y_1$, $y_2$, $e$, and $s$ have the following shapes:

$$
\small 
y_1=\left( 
\begin{array}{ccc;{1pt/1pt}cccc}
a_1&\dots&0&0&0&\dots&0\\
\vdots&\ddots&\vdots&\vdots&\vdots&\ddots&\vdots\\
0&\dots&a_k&0&0&\dots&0\\
\hdashline[1pt/1pt]
0&\dots&0&b_1-1&0&\dots&0\\
0&\dots&0&0&b_2-1&\ddots&\vdots \\
\vdots&\ddots&\vdots&\vdots&\ddots&\ddots&0 \\
0&\dots&0&0&\dots&0&b_\ell-1\\
\end{array}
\right), 
\quad e=
\left( 
\begin{array}{ccc;{1pt/1pt}cccc} 
0&\dots&0&e_{11}&e_{12}&\dots&e_{1\ell}\\
0&\dots&0&e_{21}&e_{22}&\dots&e_{2\ell}\\
\vdots&\ddots&\vdots&\vdots&\vdots&\ddots&\vdots\\
0&\dots&0&e_{k1}&e_{k2}&\dots&e_{k\ell}\\
\hdashline[1pt/1pt]
0&\dots&0&0&0&\dots&0\\
\vdots&\ddots&\vdots&\vdots&\vdots&\ddots&\vdots\\
0&\dots&0&0&0&\dots&0\\
\end{array}
\right),
$$
$$
\small 
y_2=
\left( 
\begin{array}{ccc;{1pt/1pt}cccc}
a_1-1&\dots&0&0&0&\dots&0\\
\vdots&\ddots&\vdots&\vdots&\vdots&\ddots&\vdots\\
0&\dots&a_k-1&0&0&\dots&0\\
\hdashline[1pt/1pt]
0&\dots&0&b_1&0&\dots&0\\
0&\dots&0&0&b_2&\ddots&\vdots \\
\vdots&\ddots&\vdots&\vdots&\ddots &\ddots&0\\
0&\dots&0&0&\dots & 0 &b_\ell\\
\end{array}
\right),
\quad
s=
\left( 
\begin{array}{cccc;{1pt/1pt}cccc} 
-1&0& \dots&0&s_{11}&s_{12}&\dots&s_{1\ell}\\
0&-1& \ddots&\vdots&s_{21}&s_{22}&\dots&s_{2\ell}\\
\vdots&\ddots&\ddots&0& \vdots& \vdots&\ddots&\vdots\\
0&\dots& 0 &-1&s_{k1}&s_{k2}&\dots&s_{k\ell}\\
\hdashline[1pt/1pt]
0&0& \dots&0&1&0&\dots&0\\
0&0 &\dots&0&0&1&\ddots&\vdots\\
\vdots&\vdots& \ddots&\vdots&\vdots&\ddots&\ddots&0\\
0&0& \dots&0&0&\dots&0&1\\
\end{array}
\right).
$$

The following corollary is immediate from the matrix descriptions of $y_1$, $y_2$, and $s$ given by Theorem \ref{main1}:
\begin{corollary}\label{ext}
Suppose $V$ is a $(k+\ell)$-dimensional indecomposable calibrated representation of $\sv_2$ on which $e$ does not act by $0$.  Then all simple composition factors of $V$ are $1$-dimensional, and $V$ is the following extension of semisimple $\sv_2$-modules:
$$ 0\longrightarrow\bigoplus_{i=1}^k V_{a_i}^-\longrightarrow V \longrightarrow \bigoplus_{j=1}^\ell V_{b_j-1}^+\rightarrow 0,$$
where $y_1$ acts on $V$ by $(a_1,\dots,a_k,b_1-1,\dots,b_\ell-1)\in\K^{k+\ell}$. In particular, every simple calibrated representation of $\sv_2$ is obtained from a simple calibrated representation of $\Hdeg_2$ by having $e$ act by $0$.
\end{corollary}

\subsection{Rhizomatic matrices}\label{rhizomatic} We now introduce a set of matrices that we will use for determining when a calibrated representation with regular eigenvalues is indecomposable. Let $S\in M_{k\times\ell}(\K)$ be a $k\times \ell$ matrix. 
 Define an equivalence relation on the nonzero entries of $S$ by $s_{ij}\sim  s_{mn}$ if $i=m$ or $j=n$. 
\begin{definition} 
Define the \textit{rhizomatic matrices} 
to be the set of matrices $S\in M_{k\times \ell}(\K)$ such that (i) $S$ has a single equivalence class of nonzero entries under the equivalence relation $\sim$, and (ii) $S$ has a nonzero entry in every row and column. 
\end{definition}
 \begin{example} Any matrix all of whose entries are nonzero is rhizomatic. If $\ell\geq k$ then any $k\times \ell$ matrix $S$ where $s_{ij}\neq0$ whenever $j\geq i$ is rhizomatic. Any $n\times n$ diagonal matrix, and more generally any monomial matrix, is \textit{not} rhizomatic.
 \end{example}
 \begin{example} Denote a $0$ entry by $\cdot$ and a nonzero entry by $\bullet$.
Matrix $S_1$ contains two equivalence classes of nonzero entries and is \textit{not} rhizomatic: one equivalence class has black entries $\bullet$, the other has blue entries ${\color{blue}\bullet}$. Matrix $S_2$ contains a single equivalence class but is \textit{not} rhizomatic because it has some columns and rows that are all $0$. Matrix $S_3$ contains a single equivalence class and is rhizomatic: 
$$S_1= \begin{pmatrix}
\cdot & \cdot & \cdot &{\color{blue}\bullet}  & {\color{blue}\bullet}  & \cdot & {\color{blue}\bullet}  & {\color{blue}\bullet}  & \cdot & \cdot \\
\bullet & \cdot &  \cdot & \cdot & \cdot & \cdot & \cdot & \cdot &  \cdot & \bullet \\
 \cdot & {\color{blue}\bullet}  & \cdot  & {\color{blue}\bullet}&  {\color{blue}\bullet}& {\color{blue}\bullet}& \cdot & \cdot &  \cdot & \cdot \\
 \cdot & \cdot & \cdot & \cdot & \cdot & \cdot & \cdot & \cdot & \bullet & \cdot \\
 \cdot & \cdot & \bullet & \cdot & \cdot & \cdot & \cdot & \cdot &  \cdot & \bullet \\
\cdot & {\color{blue}\bullet} &  \cdot & \cdot  & {\color{blue}\bullet} & {\color{blue}\bullet}  & {\color{blue}\bullet}  & \cdot &  \cdot & \cdot \\
 \cdot & \cdot & \bullet & \cdot & \cdot & \cdot & \cdot & \cdot & \bullet & \cdot \\
\end{pmatrix}, 
\quad\quad
S_2=\begin{pmatrix}
\cdot & \cdot & \cdot & \cdot & \cdot & \bullet & \cdot & \cdot & \cdot & \bullet\\
\cdot & \cdot & \cdot & \bullet & \bullet & \bullet & \cdot & \cdot & \cdot & \cdot\\
\cdot & \cdot & \cdot & \cdot & \cdot & \cdot & \cdot & \cdot & \cdot & \cdot\\
\cdot & \cdot & \cdot & \bullet& \cdot & \cdot & \cdot & \bullet & \bullet & \cdot\\
\bullet & \cdot & \cdot & \bullet & \cdot & \cdot & \cdot & \cdot & \cdot & \cdot\\
\cdot & \cdot & \cdot & \cdot & \cdot & \cdot & \cdot & \cdot & \cdot & \cdot\\
\cdot & \cdot & \cdot & \cdot & \cdot & \cdot & \cdot & \cdot & \cdot & \bullet\\
\end{pmatrix},$$ 
$$ S_3=\begin{pmatrix}
\bullet & \cdot & \bullet & \cdot & \cdot & \cdot &  \bullet& \cdot & \bullet & \cdot \\
\bullet & \cdot &  \cdot & \cdot & \cdot & \cdot & \cdot & \cdot &  \cdot & \bullet \\
 \cdot & \bullet & \cdot  & \cdot & \cdot & \bullet & \cdot & \bullet &  \cdot & \cdot \\
 \cdot & \cdot & \cdot & \bullet& \bullet & \cdot & \cdot & \cdot & \bullet & \cdot \\
 \cdot & \cdot & \bullet & \cdot & \bullet& \bullet & \cdot & \cdot &  \cdot & \bullet \\
\cdot & \bullet &  \cdot & \cdot & \cdot & \cdot & \cdot & \cdot &  \cdot & \cdot \\
 \bullet & \cdot & \bullet & \cdot & \cdot & \cdot & \bullet& \cdot & \bullet & \cdot \\
\end{pmatrix}. 
$$
 \end{example}

\subsection{Indecomposable calibrated representations with regular eigenvalues}

\begin{definition}Suppose $V$ is a  calibrated representation of $\sv_2$ such that $y_1-y_2$ acts on $V$ by $\begin{pmatrix}\Id_k & 0\\ 0 & -\Id_\ell
\end{pmatrix}$ in an eigenbasis for $y_1$ and $y_2$. Set $\underline{a}=(a_1,\dots,a_k)$ and $\underline{b}=(b_1,\dots,b_\ell)$ and $$\K^{k+\ell,\mathrm{reg}}=\{\ab\in\K^{k+\ell}\mid a_i\neq a_j\hbox{ and }b_m\neq b_n\hbox{ for all }1\leq i<j\leq k,\;1\leq m< n\leq \ell\}.$$ 
If $y_1$ acts on $V$ by $\diag\ab$ for some $\ab\in\K^{k+\ell,\mathrm{reg}}$ then we say that the representation $V$ has \textit{regular eigenvalues}. 
 \end{definition}
 
\begin{theorem}\label{main2} 
Suppose $V=V_{k,\ell}(S;\ab)$ is a finite-dimensional calibrated representation of $\sv_2$ on which $e$ does not act by $0$. \begin{enumerate}
\item\label{main2-1} Suppose $\ab\in\K^{k+\ell,\mathrm{reg}}$. Then $V$ is indecomposable if and only if 
$S$ is a rhizomatic matrix. 
\item\label{main2-2} Suppose $(\underline{a}, b)\in\K^{k+1}$. Then $V$ is indecomposable if and only if 
$(\underline{a}, b)\in\K^{k+1,\mathrm{reg}}$ and 
all entries of $S$ are nonzero.
\item\label{main2-3} Suppose $(a,\underline{b})\in\K^{1+\ell}$. Then $V$ is indecomposable if and only if 
$(a,\underline{b})\in\K^{1+\ell,\mathrm{reg}}$ and  
all entries of $S$ are nonzero.
\end{enumerate}
\end{theorem}

\begin{proof}

For part \eqref{main2-1}, suppose $\ab\in\K^{k+\ell,\mathrm{reg}}$. 
Recall that a representation $V$ is indecomposable if and only if $\End(V)$ is a local ring, which is equivalent to every element of $\End(V)$ being either nilpotent or invertible. We determine $\End(V)$ as follows. Let $X\in\End(V)$, so by definition $X=(x_{ij})$ is a $(k+\ell)\times (k+\ell)$ matrix that commutes with the matrices for $y_1$, $y_2$, and $s$. 
(Since $e=y_1s+1-sy_2$, we don't have to check commutation relations with $e$.) 
 First, from $y_1-y_2=\diag(1,\dots,1,-1,\dots,-1)$ it follows that $X\in M_{k\times k}(\K)\times M_{\ell\times\ell}(\K)\subset M_{(k+\ell)\times (k+\ell)}(\K)$ where we embed $M_{k\times k}(\K)$ in the upper left corner and $M_{\ell\times \ell}(\K)$ in the lower right corner of $(k+\ell)\times (k+\ell)$ matrices. Second, since $a_i\neq a_j$ for all $1\leq i<j\leq k$, and $b_m\neq b_n$ for all $1\leq m<n\leq \ell$, computing the matrix entries of the equation $y_1X=Xy_1$ shows that $X$ is a diagonal matrix, and so $$\End(V)\subseteq \{\diag(z_1,z_2,\dots,z_k,w_1,w_2,\dots,w_\ell)\mid z_i,w_j\in\K\}\cong \K^{k+\ell}.$$ 
(Computing the commutator of $X$ with $y_2$ now gives no new information, since we already used $y_1$ and $y_1-y_2$).

Write $X=\diag(z_1,z_2,\dots,z_k,w_1,w_2,\dots,w_\ell)$. We now determine $\End(V)$ as a subalgebra of diagonal matrices using the remaining equation $Xs-sX=0$. Computing the commutator $Xs-sX$, all entries are automatically $0$ except in the upper right $k\times \ell$ corner block, where we obtain the following $k\times \ell$ entries:
$$\begin{pmatrix}s_{11}(z_1-w_1)& s_{12}(z_1-w_2)&\dots &s_{1\ell}(z_1-w_\ell)\\
s_{21}(z_2-w_1)&s_{22}(z_2-w_2)&\dots &s_{2\ell}(z_2-w_\ell)\\
\vdots &\vdots&\ddots&\vdots\\
s_{k1}(z_k-w_1)&s_{k2}(z_k-w_2)&\dots&s_{k\ell}(z_k-w_\ell)
\end{pmatrix}. 
$$
Since $Xs-sX=0$, each of these $k\ell$ entries is equal to $0$. Thus for a given pair $(i,j)$, either $s_{ij}=0$ or $z_i=w_j$. Taking the equivalence class of a nonzero entry $s_{ij}$ as in Section \ref{rhizomatic}, it follows that $z_r=w_s=z_i=w_j$ for all $s_{rs}\sim s_{ij}$, i.e. all the $z_r$'s and $w_s$'s are equal to each other such that $r$ is the row or $s$ is the column of some nonzero entry $s_{rs}\sim s_{ij}$. If $s_{ij}$ and $s_{rs}$ are in different equivalence classes, then there is no relation between $z_i$ and $z_r$ or between $w_j$ and $w_s$. And finally, if some row $r$ contains all $0$ entries then we get no relation on $z_r$; similarly, if some column $s$ contains all $0$ entries then we get no relation on $w_s$. Let $n(S)\geq 1$ be the number of equivalence classes of nonzero entries of $S$, let $Z_r$ be the number of rows of $S$ that contain only $0$'s, and let $Z_c$ be the number of columns that contain only $0$'s. We have that $\End(V)\cong \K^{n(S)+Z_r+Z_c}$, but $\K^{n(S)+Z_r+Z_c}$ is a local ring if and only if $n(S)+Z_r+Z_c=1$ if and only if $n(S)=1$ and $Z_r=Z_c=0$ if and only if $S$ is rhizomatic. 
This concludes the proof of part \eqref{main2-1}.

We turn now to part \eqref{main2-2}. One direction of the statement is simply a special case of part \eqref{main2-1} when $\ell=1$: if $S$ is a $k\times 1$ matrix then $S$ is rhizomatic if and only if all the entries of $S$ are nonzero, thus if all entries of $S$ are nonzero and $(\underline{a}, b)\in\K^{k+1,\mathrm{reg}}$ then part \eqref{main2-1} says that $V$ is indecomposable.  For the converse direction, suppose that $V$ is indecomposable. If some entry $s_{i1}$ of $S$ is $0$ 
 then we see that the actions of the generators of $\sv_2$ preserve the subspaces $\K \mathbf{v}_i$ and $\K \mathbf{v}_1\oplus\dots \K \mathbf{v}_{i-1}\oplus \K \mathbf{v}_{i+1}\dots \K \mathbf{v}_{k+1}$ (where $\mathbf{v}_i$ denotes the $i$'th basis vector $(0,\dots,1,\dots,0)$ of $\K^{k+1}$ with $1$ in the $i$'th place and $0$'s elsewhere); thus $V$ splits as a direct sum of these two subrepresentations contradicting the assumption that $V$ is indecomposable. So $s_{i1}\neq 0$ for all $i=1,\dots,k$. Suppose $(\underline{a}, b)\notin \K^{k+1,\mathrm{reg}}$, so $a_i=a_m$ for some $i\neq m$; without loss of generality we may assume $i=1$ and $m=2$. Then the centralizer of $y_1$ and $y_2$ contains any matrix of the form 
$$
\small 
X=
\left( 
\begin{array}{cc;{1pt/1pt}cccc}
x_{11}&x_{12}&0&0& \dots&0\\
x_{21}&x_{22}&0&0& \dots&0\\
\hdashline[1pt/1pt]
0&0&x_{33}&0& \dots &0\\
0&0&0&x_{44}& \ddots &\vdots \\
\vdots &\vdots & \vdots& \ddots &\ddots&0\\
0&0&0&\dots  & 0 & x_{k+\ell,k+\ell}
\end{array}
\right).
$$
Computing the matrix entries of the equation $Xs-sX=0$ we get the following $k$ entries which are all equal to $0$:
$$
\begin{pmatrix}
(x_{11}-x_{k+1,k+1})s_{11}+x_{12}s_{21}\\
x_{21}s_{11}+(x_{22}-x_{k+1,k+1})s_{21}\\
(x_{33}-x_{k+1,k+1})s_{31}\\
\vdots\\
(x_{k,k}-x_{k+1,k+1})s_{k1}
\end{pmatrix}.
$$
Since all $s_{i1}\neq 0$ for $i=1,\dots,k$, $x_{i,i}=x_{k+1,k+1}$ for all $i=3,\dots, k$, and from the first and second lines we get that $x_{k+1,k+1}$ can be solved in terms of $x_{21}, x_{22}, s_{11},$ and $s_{21}$, and then similarly we can solve for $x_{12}$ in terms of $x_{11},\;x_{21},\;x_{22},\;s_{11},$ and $s_{21}$ in the first equation. Then we have
$$\End(V)\cong
\begin{pmatrix}
\K &0\\
\K & \K
\end{pmatrix}, 
$$
which is not a local ring, contradicting the assumption that $V$ is indecomposable. Therefore all the eigenvalues $a_i$ are distinct, and part \eqref{main2-2} is proved.

Finally, part \eqref{main2-3} is proved in a totally symmetric way to part \eqref{main2-2}.
\end{proof}
\begin{remark}
In fact, if $\ab\in\K^{k+\ell,\mathrm{reg}}$ and $S$ is rhizomatic, then for some $(i,j)$ the entry $e_{ij}=(a_i-b_j)s_{ij}$ of $e$ is automatically nonzero. Indeed, by way of contradiction suppose that $e$ is the $0$ matrix, but $S$ is rhizomatic. Then for any $s_{ij}\neq 0$ there is some other $s_{ik}\neq 0$ in the same row or some other $s_{hj}\neq 0$ in the same column. In the first case, $(a_i-b_j)s_{ij}=0=(a_i-b_k)s_{ik}$ forces $b_j=a_i$ and $b_k=a_i$ and thus $b_j=b_k$ for some $j\neq k$, contradicting the assumption $\ab\in\K^{k+\ell,\mathrm{reg}}$. Similarly in the second case. Thus $\ab\in\K^{k+\ell,\mathrm{reg}}$ and $S$ rhizomatic implies that $e$ does not act by $0$ on $V_{k,\ell}(S;\ab)$.
\end{remark}

\begin{example}
Let $k=3$ and $\ell=2$ and take $\ab=(2i,-2i,1,-1,1)$. Then $\ab\in\mathbb{C}^{3+2,\mathrm{reg}}$. Take $S=\begin{pmatrix}0&1\\-\pi&5\\\frac{ i\pi}{2}&0\end{pmatrix}$, a rhizomatic matrix. Then $V_{3,2}(S;\ab)$ is an indecomposable calibrated $\sv_2$-representation by Theorem \ref{main2-1}, and $y_1,y_2,s,e$ act by:
\begin{align*}
&y_1=
\left( 
\begin{array}{ccc;{1pt/1pt}cc}
2i&0&0&0&0\\
0&-2i&0&0&0\\
0&0&1&0&0\\
\hdashline[1pt/1pt]
0&0&0&-2&0\\
0&0&0&0&0
\end{array}
\right),
\quad
y_2=
\left( 
\begin{array}{ccc;{1pt/1pt}cc}
2i-1&0&0&0&0\\
0&-2i-1&0&0&0\\
0&0&0&0&0\\
\hdashline[1pt/1pt]
0&0&0&-1&0\\
0&0&0&0&1
\end{array}
\right),\\
&s=
\left( 
\begin{array}{ccc;{1pt/1pt}cc}
-1&0&0&0&1\\
0&-1&0&-\pi&5\\
0&0&-1&\frac{i\pi}{2}&0\\
\hdashline[1pt/1pt]
0&0&0&1&0\\
0&0&0&0&1
\end{array}
\right),
\quad
e=
\left(
\begin{array}{ccc;{1pt/1pt}cc}
0&0&0&0&1-2i\\
0&0&0&\pi-1&5+10i\\
0&0&0&1&0\\
\hdashline[1pt/1pt]
0&0&0&0&0\\
0&0&0&0&0
\end{array}
\right).
\end{align*}
\end{example}

\subsection{Isomorphism classes of calibrated representations with regular eigenvalues} Many of the calibrated representations we constructed in the previous section may be isomorphic to each other. In this section, we determine when two indecomposable calibrated representations with regular eigenvalues are isomorphic.

 Let $(\underline{a},\underline{b})\in\K^{k+\ell,\mathrm{reg}}$, so $\underline{a}=(a_1,\dots,a_k)$ with $a_i\neq a_j$ for all $i\neq j$, and $\underline{b}=(b_1,\dots,b_\ell)$ with $b_m\neq b_n$ for all $m\neq n$. Let $N_k\subset\GL_k(\K)$ be the normalizer of diagonal $k\times k$ matrices, so $N_k$ is the group of $k\times k$ matrices that have exactly one nonzero entry in every row and column. Similarly, let $N_\ell$ be the group of $\ell\times\ell$ matrices with one nonzero entry in every row and column. 
\begin{theorem}\label{iso1} The group $N_k\times N_\ell$ acts naturally on the set of $(k+\ell)$-dimensional indecomposable calibrated representations with regular eigenvalues and on the space of pairs consisting of a $(k\times \ell)$ rhizomatic matrix and a vector $\ab\in \K^{k+\ell,\mathrm{reg}}$ which parametrize these representations: $$\left(N_k\times N_\ell\right) \curvearrowright \widetilde{\mathcal{V}}_{k,\ell}:=\{(S,\ab)\mid S\in M_{k\times\ell}(\K) \hbox{ is rhizomatic},\;\ab\in\K^{k+\ell,\mathrm{reg}}\}.$$
\end{theorem}
\begin{proof}
We embed $N_k\times N_\ell$ in $\GL_{k+\ell}(\K)$ in the obvious way, as block matrices:
$$N_k\times N_\ell\cong\begin{pmatrix} N_k & 0\\
0& N_\ell\end{pmatrix},$$
which then act by conjugation on $y_1,y_2,s,e$. On the matrix $S$, $N_k$ acts on the left via the left multiplication while $N_\ell$ acts on the right via the right (inverse) multiplication. We also call these actions left translation and right (inverse) translation. Take elements $X_1=\sum_{i=1}^k\xi_ie_{i,\sigma(i)}\in N_k$ and $X_2=\sum_{j=1}^\ell \phi_je_{j,\tau(j)}\in N_\ell$ where $\xi_i,\phi_j\in\K^\times$, $\sigma\in S_k$, $\tau\in S_\ell$. On the matrices for $y_1, y_2, s, e$ we get the following effect:
\begin{align*}
&(X_1,X_2)\cdot y_1=\diag(a_{\sigma(1)},a_{\sigma(2)},\dots,a_{\sigma(k)},b_{\tau(1)}-1,b_{\tau(2)}-1,\dots, b_{\tau(\ell)}-1),\\
&(X_1,X_2)\cdot y_2=\diag(a_{\sigma(1)}-1,a_{\sigma(2)}-1,\dots,a_{\sigma(k)}-1,b_{\tau(1)},b_{\tau(2)},\dots, b_{\tau(\ell)}),\\
&(X_1,X_2)\cdot s=\begin{pmatrix}X_1&0\\0&X_2\end{pmatrix} \begin{pmatrix}-\Id_k&S\\0&\Id_\ell\end{pmatrix}\begin{pmatrix}X_1^{-1}&0\\0&X_2^{-1}\end{pmatrix}=\begin{pmatrix} -\Id_k & X_1SX_2^{-1}\\ 0&\Id_\ell\end{pmatrix},\\
&(X_1,X_2)\cdot e=\begin{pmatrix}X_1&0\\0&X_2\end{pmatrix} \begin{pmatrix}0&E\\0&0\end{pmatrix}\begin{pmatrix}X_1^{-1}&0\\0&X_2^{-1}\end{pmatrix}=\begin{pmatrix} 0& X_1EX_2^{-1}\\ 0&0\end{pmatrix}.\\
\end{align*}

 The action on pairs $(S,\ab)$ is given explicitly by: 
$$
(X_1,X_2)\cdot (S,\ab)=( X_1SX_2^{-1},( a_{\sigma(1)},\dots  ,  a_{\sigma(k)},   b_{\tau(1)},\dots,b_{\tau(\ell)} )),
$$
where the effect of the action $S\rightarrow X_1SX_2^{-1}$ is to permute the rows of $S$ by $\sigma$ and the columns by $\tau^{-1}$, then to multiply the $i$'th row of the resulting matrix by $\xi_i$ and the $j$'th column by $\phi_j$. That is, we have $\left(X_1SX_2^{-1}\right)_{ij}=\xi_i\phi_js_{\sigma^{-1}(i),\tau(j)}$. 

We need to check that $X_1SX_2^{-1}$ is also a rhizomatic matrix. The minimal relations generating the relation $s_{ij}\sim s_{rp}$ in Section \ref{rhizomatic} are of the form $s_{ij}\sim s_{ir}$ and $s_{ij}\sim s_{pj}$, i.e. the relations given by two nonzero entries being in the same row, and two nonzero entries being in the same column. Since $\xi_i,\phi_j\neq 0$, $\left(X_1SX_2^{-1}\right)_{ij}\neq 0$ if and only if $s_{\sigma^{-1}(i),\tau(j)}\neq 0$. Therefore $\left(X_1SX_2^{-1}\right)_{ij}\sim \left(X_1SX_2^{-1}\right)_{ir}$ if and only if $s_{\sigma^{-1}(i),\tau(j)}\sim s_{\sigma^{-1}(i),\tau(r)}$, and $\left(X_1SX_2^{-1}\right)_{ij}\sim \left(X_1SX_2^{-1}\right)_{pj}$  if and only if $s_{\sigma^{-1}(i),\tau(j)}\sim s_{\sigma^{-1}(p),\tau(j)}$. It then follows that $\left(X_1SX_2^{-1}\right)_{ij}\sim \left(X_1SX_2^{-1}\right)_{pr}$ if and only if $s_{\sigma^{-1}(i),\tau(j)}\sim s_{\sigma^{-1}(p),\tau(r)}$. Thus $X_1SX_2^{-1}$ has a single equivalence class of entries since $S$ does. Since a nonzero entry appears in every row and column of $S$, the same is true for $X_1SX_2^{-1}$. Therefore $X_1SX_2^{-1}$ is again rhizomatic, and we indeed get an action.
\end{proof}

\begin{theorem}\label{iso2}
Let $V_1,V_2\in \widetilde{\mathcal{V}}_{k,\ell}$ be two indecomposable calibrated representations of dimension $k+\ell$ with regular eigenvalues. Then $V_1\cong V_2$ as $\sv_2$-representations if and only if $V_1$ and $V_2$ are in the same $(N_k\times N_\ell)$-orbit. Thus $\mathcal{V}_{k,\ell}:=\widetilde{\mathcal{V}}_{k,\ell}/(N_k\times N_\ell)$ parametrizes the isomorphism classes of indecomposable $(k+\ell)$-dimensional calibrated $\sv_2$-representations with regular eigenvalues.
\end{theorem}

\begin{proof} Let the matrices of the generators $y_1,y_2,s$ acting on $V_i$, $i=1,2$, be given by 
\begin{align*}
y_1^{(i)}&=\diag(a_1^{(i)},\dots, a_k^{(i)},b_1^{(i)}-1,\dots,b_\ell^{(i)}-1),\\
y_2^{(i)}&=\diag(a_1^{(i)}-1,\dots, a_k^{(i)}-1,b_1^{(i)},\dots,b_\ell^{(i)}),\\
s^{(i)}&=\begin{pmatrix}-\Id_k & S^{(i)}\\0&\Id_\ell\end{pmatrix}.
\end{align*}
By definition, $V_1\cong V_2$ if and only if there exists $A\in\GL_{k+\ell}(\K)$ such that 
$$
A\left(y_1^{(1)}\right)A^{-1}=y_1^{(2)},\qquad
A\left(y_2^{(1)}\right)A^{-1}=y_2^{(2)},\qquad
A\left(s^{(1)}\right)A^{-1}=s^{(2)}.
$$
The first two equations imply that $A\in N_k\times N_\ell$. Thus $V_1\cong V_2$ implies that $V_1$ and $V_2$ are in the same $(N_k\times N_\ell)$-orbit. The converse follows from the previous theorem which showed that conjugation by $N_k\times N_\ell$ respects the $\sv_2$-action.
\end{proof}

\section*{Acknowledgments}{\footnotesize
This material is based upon work supported by the National Science Foundation under Grant No. DMS-1440140, and the National Security Agency under Grant No. h98230-18-1-0144 while the authors were in residence at the Mathematical Sciences Research Institute in Berkeley, California, during the Summer of 2018.  We thank Vera Serganova and Inna Entova-Aizenbud for stimulating discussions during that visit. 
}

\end{document}